\documentclass[review,12pt]{elsarticle}
\usepackage[T1]{fontenc}
\usepackage[utf8]{inputenc}
\usepackage{amsmath}
\usepackage{amssymb}
\usepackage{bm}  
\usepackage{graphicx}
\usepackage{float}
\usepackage{epsfig}
\usepackage[toc,page]{appendix}
\usepackage{caption}
\usepackage{amsthm}

\newcommand{\be}{\begin{equation}}
\newcommand{\ee}{\end{equation}}
\newcommand{\bee}{\begin{equation*}}
\newcommand{\eee}{\end{equation*}}

\newtheorem{theorem}{Theorem}

\newtheorem{definition}[theorem]{Definition}

\title{On graph Laplacian eigenvectors with components in $\{-1,0,1\}$} 

\author{J. G. Caputo \fnref{fn1}}
\ead{caputo@insa-rouen.fr}
\author{I. Khames \fnref{fn2}}
\ead{imene.khames@insa-rouen.fr}
\author{A. Knippel \fnref{fn3}}
\ead{arnaud.knippel@insa-rouen.fr}

\address{Laboratoire de Mathématiques, INSA Rouen Normandie. 76800 Saint-Etienne du Rouvray, FRANCE.}

\fntext[fn1]{caputo@insa-rouen.fr}
\fntext[fn2]{imene.khames@insa-rouen.fr}
\fntext[fn3]{arnaud.knippel@insa-rouen.fr}

\begin{document}

\begin{abstract}

We characterize all graphs for which there are eigenvectors of the
graph Laplacian having all their components 
in $\{-1,+1 \}$ or $\{-1,0,+ 1 \}$. Graphs having eigenvectors 
with components
in $\{-1,+1 \}$ are called bivalent and are shown to be 
the regular bipartite graphs and their extensions obtained by 
adding edges between vertices with the same value for the 
given eigenvector.
Graphs with eigenvectors with components in $\{-1,0,+ 1 \}$ are
called trivalent and are shown to be soft-regular graphs --graphs
such that vertices associated with non-zero components have the same
degree -- and their extensions via some transformations.

\end{abstract}

\maketitle

\section{Introduction}

The graph wave equation \cite{cks13}, 
where the Laplacian is replaced by the graph Laplacian \cite{cvetkovic},
is a natural model for describing miscible flows on a network since
it arises from conservation laws. The graph wave equation is well understood in terms of normal modes
i.e. periodic solutions associated with the eigenvectors of the graph Laplacian.
In a previous work \cite{ckkp17},
we considered a nonlinear graph wave equation with a cubic on-site nonlinearity 
which is the discrete $\phi^4$ model \cite{scott}
and we studied the extension of the normal modes into nonlinear periodic orbits.

We generalized the criterion of Aoki \cite{aoki} for paths and cycles to the case of general graphs
and showed that the linear normal modes associated with eigenvectors 
composed from $\{-1,0,1\}$ extend to nonlinear periodic orbits.
We defined monovalent, bivalent and trivalent eigenvectors 
depending whether their components are in 
$\{+1\}$ or $\{-1,+1\}$ or $\{-1,0,+1\}$. The first case is trivial 
as the all-1 vector is always an eigenvector of the graph 
Laplacian, associated with the eigenvalue 0.

The trivalent eigenvectors contain components of value 0, corresponding 
to vertices that we call soft nodes to emphasize their special 
role in the dynamical systems, as analyzed in \cite{cks13}.  
A classification of graphs whose Laplacian matrices have 
eigenvectors with soft nodes is presented in \cite{ck16}.

In \cite{ckkp17}, we classified the bivalent and trivalent eigenvectors in paths and cycles 
for which the spectrum is well-known \cite{edwards}. It is then natural to try and characterize the graphs having bivalent and trivalent eigenvectors. \\

Wilf \cite{wilf} asked 
what kind of graph admits an adjacency matrix eigenvector
consisting solely of $\pm 1$ entries.
More recently, Stevanovi\'c \cite{stevanovic} proved that Wilf's problem is 
NP-complete,
and that the set of graphs having a $\pm 1$ eigenvector of 
the adjacency matrix is quite rich.

We ask here the same question in the case of the Laplacian matrix 
of a graph, and give a characterization of graphs having Laplacian 
eigenvectors with components in $\{-1, 1\}$ or $\{-1, 0, 1\}$. We 
call these graphs respectively bivalent and trivalent. This is done 
using transformations of graphs, three from the literature 
\cite{merris98} and one of our own.
In the case of regular graphs, all results about the Laplacian 
spectrum of graphs carry over to results about the adjacency spectrum. \\

The article is organized as follows. In section 2,
we introduce some preliminaries of the graph Laplacian
and transformations of graphs.
Section 3 presents a characterization of bivalent graphs : we show that the bivalent graphs are the regular bipartite graphs 
and their extensions by 
adding edges between two equal-valued vertices.
Section 4 presents a similar characterization for trivalent graphs : we show that the trivalent graphs are obtained from what we call soft regular graphs by applying some
transformations.

\section{Graph Laplacian}
Let $\mathcal{G}(\mathcal{V},\mathcal{E})$ be a graph with 
vertex set $\mathcal{V}$ of cardinality $N$
and edge set $\mathcal{E}$.
All graphs in this article are finite and undirected with no loops or multiple edges.
Denote the degree of vertex $j$ by $d_j$ and let 
$D$ be the $N\times N$ diagonal matrix of vertex degrees $D_{jj}=d_j$.
We will indicate adjacency of vertices by $i\sim j$ for $e_{ij} \in \mathcal{E}(\mathcal{G})$.
Let $A$ be the $N\times N ~ \left\{0,1 \right\}$ adjacency matrix such that $A_{ij}=1$ if and only if 
$e_{ij} \in \mathcal{E}(\mathcal{G})$ ($i\neq j$).
The Laplacian matrix \cite{cvetkovic} associated with the graph 
$\mathcal{G}$ is
the matrix $\Delta=D-A$. 
For an extensive survey on the Laplacian matrix see Merris \cite{merris94}.

Since the graph Laplacian $\Delta$ is a real symmetric positive 
semi-definite matrix, it is diagonalizable, say
\begin{equation}
\label{eigenvectors}
\Delta v^{k}=\lambda_k v^{k},
\end{equation}
where the eigenvectors $v^{k},~k\in \{1,\dots,N\}$ of $\Delta$ can be chosen 
to be orthogonal with respect to the scalar product in $\mathbb{R}^{N}$.
We arrange the eigenvalues $\lambda_k$ of $\Delta$ as
$\lambda_1=0 \leq  \lambda_2 \leq \cdots \leq \lambda_N$.
The first eigenvalue $\lambda_1=0$ corresponds to the monovalent eigenvector $v^1 = (1,1, \dots , 1 )^T$.

We refer to the Laplacian $\Delta$ of the graph $\mathcal{G}$ as $\Delta(\mathcal{G})$.
Thus, $v$ is an eigenvector of $\Delta(\mathcal{G})$ affording $\lambda$ if and only if 
\begin{equation}
\label{eigenvector}
(d_i-\lambda)v_i=\sum_{j\sim i}v_j, ~~~ \forall i\in \{1,\dots,N\} .
\end{equation}

\subsection{Definitions}

\begin{definition}[Soft node \cite{cks13}]
A node $j$ of a graph is a soft node for an eigenvalue 
$\lambda$ of the graph Laplacian if there exists an eigenvector 
$v$ for this eigenvalue such that $v_j = 0$.
\end{definition}

\begin{definition}[Regular graph]
A graph is $d$-regular if every vertex has the same degree $d$.
\end{definition}

\begin{definition}[Soft regular graph]
A graph is $d$-soft regular for an eigenvector $v$ of the Laplacian
if every non-soft node for $v$ has the same degree $d$.
\end{definition}

The graph on the left of Fig.\ref{soft} is $3$-soft regular 
for the eigenvector \\
$(0,1,1,0,-1,-1)^T$ since all the non-zero vertices 
have the same degree 3.
The graph on the right of Fig.\ref{soft} is non-soft regular for
the eigenvector \\
$(0,1,1,0,-1,-1,0,0)^T$ since the non-zero vertices
have different degrees.

\begin{figure}[H]
\centerline{\resizebox{10 cm}{5 cm}{\includegraphics{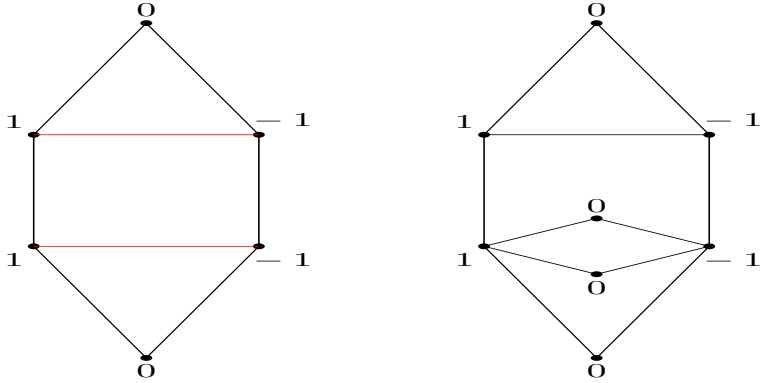}}}
\caption{\label{soft} $3$-soft regular graph for the Laplacian eigenvector $\left(0,1,1,0,-1,-1 \right)^T$ (left). 
Non-soft regular graph for the Laplacian eigenvector $\left(0,1,1,0,-1,-1,0,0 \right)^T$ (right).}
\end{figure}

\begin{definition}[Bivalent graph]
A graph is bivalent if there exists an eigenvector of the graph Laplacian 
composed from $-1,+1$.
Such a vector is called bivalent.
\end{definition}

The bivalent eigenvector $v$ must have as many $-1$ and $+1$ components,
and thus the bivalent graph must have an even number of nodes.
This is a consequence of the orthogonality of $v$ to the monovalent eigenvector $v^1$.

\begin{definition}[Trivalent graph]
A graph is trivalent if there exists an eigenvector of the graph Laplacian 
composed from $-1,0,+1$.
Such a vector is called trivalent.
\end{definition}

\begin{definition}[$k$-partite graph]
A $k$-partite graph is a graph whose vertices can be partitioned into $k$ different independent sets
so that no two vertices within the same set are adjacent.
\end{definition}

When $k = 2$ these are the bipartite graphs, and when $k = 3$ they are 
the tripartite graphs.

\begin{definition}[Perfect matching]
A perfect matching of a graph $\mathcal{G}$ is a matching (i.e., an independent edge set) 
in which every vertex of the graph is incident to exactly one edge of the matching. 
\end{definition}

\begin{definition}[Alternate perfect matching]
An alternate perfect matching for a vector $v$ on the nodes of a graph $\mathcal{G}$
is a perfect matching for the nonzero nodes such that edges $e_{ij}$ of the matching
satisfy $v_i=-v_j ~~ ( \neq 0)$.
\end{definition}

The left of Fig.\ref{soft} shows 
the alternate perfect matching (represented by red lines) for the eigenvector $(0,-1,-1,0,1,1)^T$ on the nodes of the 6-cycle.

\subsection{Transformations of graphs}
\label{principe}
Merris \cite{merris98} considers several transformations of graphs based on Laplacian eigenvectors.
In the following we review three of them and we present another transformation.

\subsubsection{Transformations preserving eigenvalues}
\begin{theorem}[\textbf{Link between two equal nodes} \cite{merris98}] 
\label{L2S}
Let $v$ be an eigenvector of $\Delta(\mathcal{G})$ affording an eigenvalue $\lambda$.
If $v_i=v_j$, then $v$ is an eigenvector of $\Delta(\mathcal{G}')$ affording the eigenvalue $\lambda$,
where $\mathcal{G}'$ is the graph obtained from $\mathcal{G}$ by deleting or adding the edge $e_{ij}$
depending whether $e_{ij}$ is an edge of $\mathcal{G}$ or not.
\end{theorem}

Fig.\ref{figL2S} shows how Theorem \ref{L2S}
can be used to extend an eigenvector and its eigenvalue to the transformed graphs
by adding edges (represented by red lines) between nodes having the same value.
Notice that this transformation does not preserve the regularity of the graph.

\begin{figure}[H]
\centerline{\resizebox{12 cm}{3 cm}{\includegraphics{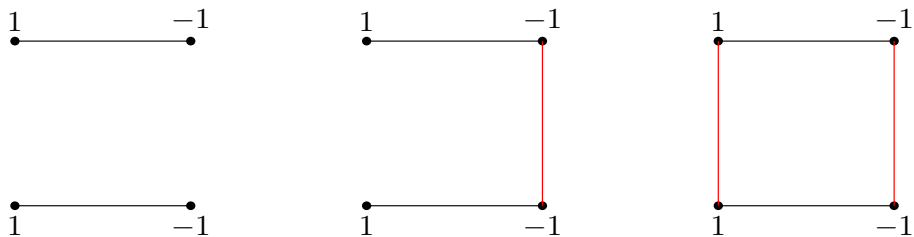}}}
\caption{\label{figL2S} Three graphs obtained by adding or deleting edges between equal nodes,
affording (the same eigenvalue) $\lambda=2$.}
\end{figure}

\begin{theorem}[\textbf{Extension of soft nodes} \cite{merris98}]
\label{extension}
For a graph $\mathcal{G}(\mathcal{V},\mathcal{E})$ fix a nonempty 
subset $\mathcal{W}$ of $\mathcal{V}$. Let $\mathcal{G}(\mathcal{W})$
be the graph obtained by 
removing all the vertices in $\mathcal{V} \backslash \mathcal{W}$ 
that are adjacent in $\mathcal{G}$
to no vertex of $\mathcal{W}$ and
any remaining edge that is incident with no vertex of $\mathcal{W}$.
Suppose $v$ is an eigenvector of the Laplacian of the reduced graph $\mathcal{G}\{\mathcal{W}\}$
that affords $\lambda$ and is supported by $\mathcal{W}$ in the sense that if
$v_i \neq 0$, then $i\in \mathcal{W}$.
Then the extension $v'$ with $v'_j=v_j$ for $j\in \mathcal{W}$ and $v'_j=0$ otherwise
is an eigenvector of $\Delta(\mathcal{G})$ affording $\lambda$.
\end{theorem}

\begin{figure}[H]
\centerline{\resizebox{8 cm}{3 cm}{\includegraphics{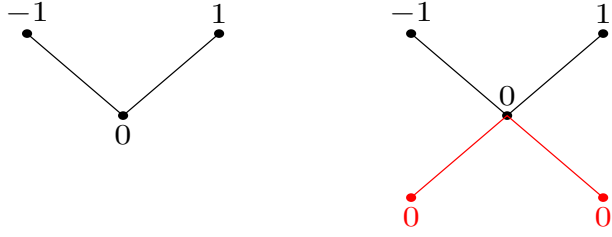}}}
\caption{\label{articulation} Extension at soft node of the eigenvector $\left(-1,0,1\right)^T$
by adding soft nodes. The eigenvectors afford (the same eigenvalue) $\lambda=1$.}
\end{figure}

We introduce the following transformation 
which preserves the eigenvalues and does not preserve the soft regularity of the graph.

\begin{theorem}[\textbf{Replace an edge by a soft square}]
\label{square}
Let $v$ be an eigenvector of $\Delta(\mathcal{G})$ affording an eigenvalue $\lambda$.
Let $\mathcal{G}'$ be the graph obtained from $\mathcal{G}$
by deleting an edge $e_{ij} \in \mathcal{E}(\mathcal{G})$ such that $v_i=-v_j$
and adding two soft nodes $k,l \in \mathcal{V}(\mathcal{G}')$ 
for the extension $v'$ of $v$ 
(such that $v'_m=v_m$ for $m\in \mathcal{V}(\mathcal{G})$ 
and $v'_k=v'_l=0$)
and adding four edges $e_{ik},e_{kj},e_{il},e_{lj} \in \mathcal{E}(\mathcal{G'})$. 
Then, $v'$ is an eigenvector of $\Delta(\mathcal{G}')$ for the eigenvalue $\lambda$.
\end{theorem}

\begin{proof}
Suppose the edge $e_{ij} \in \mathcal{E}(\mathcal{G})$ joining two nodes having opposite values $v_i=-v_j$,
is replaced by a square 
$e_{ik},e_{kj},e_{il},e_{lj} \in \mathcal{E}(\mathcal{G}')$
of soft nodes $k,l \in \mathcal{V}(\mathcal{G'})$.
The eigenvector condition
$$\left((d_i+1)-\lambda\right)v_i=v_i+\sum_{ m\sim i} v_m =\sum_{ m\sim i,~m\neq j} v_m=2\times 0+\sum_{ m\sim i,~m\neq j} v_m,$$
is the condition that must be met at vertex $i$ 
for the extension $v'$ of $v$, by defining $v'_m=0$ for $m\in \mathcal{V}(\mathcal{G}')\backslash
\mathcal{V}(\mathcal{G})$, to be an eigenvector of $\Delta(\mathcal{G}')$ affording $\lambda$.
The eigenvector condition at vertex $j$ is confirmed similarly, 
and the conditions at the other vertices are the same for $\mathcal{G}'$ as they are for $\mathcal{G}$.
\end{proof}

Fig.\ref{square_fig} shows how Theorem \ref{square}
can be used to transform a soft regular graph to a non-soft regular graph
without changing the eigenvalue.
Note that a square of soft nodes can be replaced by an edge between opposite nodes.

\begin{figure}[H]
\centerline{\resizebox{8 cm}{5 cm}{\includegraphics{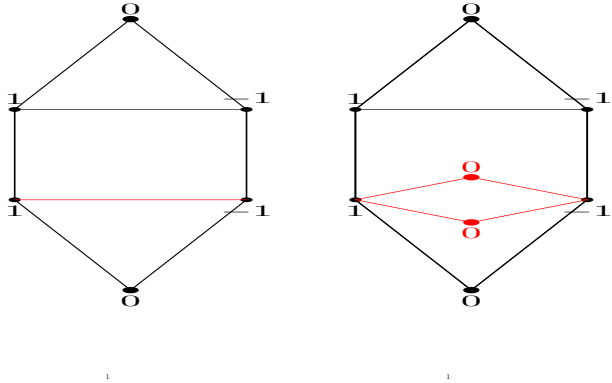}}}
\caption{\label{square_fig} Replacing an edge between opposite nodes by a square of soft nodes.
The eigenvectors afford (the same eigenvalue) $\lambda=3$.}
\end{figure}

\subsubsection{Transformations changing eigenvalues}
The following transformation allows us to extend graphs
by changing the eigenvalues
and preserving the soft regularity of the graph.

\begin{theorem}[\textbf{Add/Delete an alternate perfect matching} \cite{merris98}] 
\label{alt}
Let $v$ be an eigenvector of $\Delta(\mathcal{G})$ affording an eigenvalue $\lambda$.
Let $\mathcal{G}'$ be the graph obtained from $\mathcal{G}$
by adding (resp. deleting) an alternate perfect matching for $v$.
Then, $v$ is an eigenvector of $\Delta(\mathcal{G}')$ affording the eigenvalue $\lambda +2$ (resp. $\lambda -2$).
\end{theorem}

Adding an alternate perfect matching is illustrated in Fig.\ref{alt_fig}.
This transformation preserves the soft regularity of the graph 
and increases the eigenvalue by 2.

\begin{figure}[H]
\centerline{\resizebox{13 cm}{4 cm}{\includegraphics{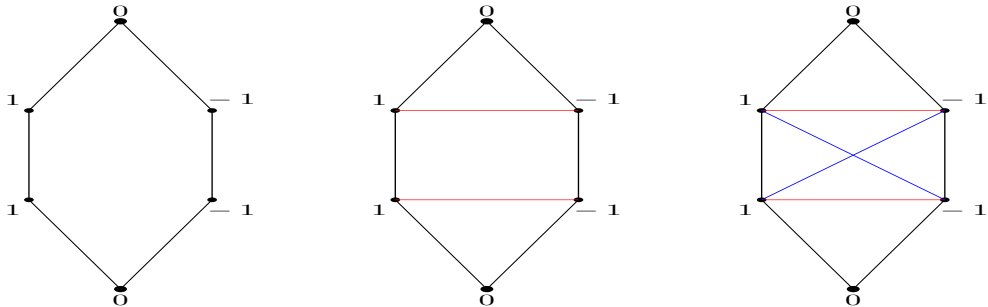}}}
\caption{\label{alt_fig}
Graphs obtained by adding an alternate perfect matching for the eigenvector
$\left(0,1,1,0,-1,-1 \right)^T$.
The eigenvalues are
$\lambda=1$ (left), $\lambda=3$ (middle) and $\lambda=5$ (right).}
\end{figure}

\section{Bivalent graphs}

For bivalent graphs, we give the following characterization.

\begin{theorem}[\textbf{Bivalent graphs}]
The bivalent graphs are the regular bipartite graphs and their extensions obtained
by adding edges between nodes having the same value for a bivalent eigenvector.
\end{theorem}

\begin{proof}
Let $\mathcal{G}$ be a graph having a bivalent eigenvector $v$ affording $\lambda$.
We reduce $\mathcal{G}$ by deleting all 
the edges between equal nodes Theorem \ref{L2S},
thus obtaining a graph where edges only connect $+1$ to $-1$.
This is a bipartite graph.

We write the eigenvector condition 
for nodes $j$ (with degree $d_j$) such that $v_j=1$ 
\be
\label{biv}
(d_j)(1) + \sum_{i\sim j} (-1)(-1) = 2d_j=\lambda .
\ee
A similar equation holds for nodes $j$ such that $v_j=-1$.

The eigenvector condition for all vertices
of $\mathcal{G}$ requires that $\lambda=2d_j,~ \forall j \in \{1,\dots,N \}$
so that $d_j=d,~ \forall j \in \{1,\dots,N \}$. 
Thus, $\mathcal{G}$ is $d$-regular graph.
Hence, a bivalent graph is either a $d$-regular bipartite graph
or obtained from such a graph by 
adding edges between equal nodes Theorem \ref{L2S}.

Conversely, if $\mathcal{G}$ is a bipartite $d$-regular graph 
and $\mathcal{G}'$ is obtained from $\mathcal{G}$ by
adding edges between equal nodes then 
(\ref{biv}) is satisfied and $\mathcal{G}'$ is bivalent.
\end{proof}

As an example, Fig.\ref{bivalent} shows the smallest bivalent graph, with eigenvalue $\lambda=2$. It is a 1-regular graph.

\begin{figure}[H]
\centerline{\resizebox{2.8 cm}{0.5 cm}{\includegraphics{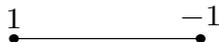}}}
\caption{\label{bivalent} A 1-regular bivalent graph ($d=1, ~~\lambda=2$).}
\end{figure}

The extension of two copies of chain of length 1 seen in Fig.\ref{bivalent}
by adding an alternate perfect matching Theorem \ref{alt}
produces the $2$-regular bivalent graph shown on the right of Fig.\ref{2reg}

\begin{figure}[H]
\centerline{\resizebox{8 cm}{3 cm}{\includegraphics{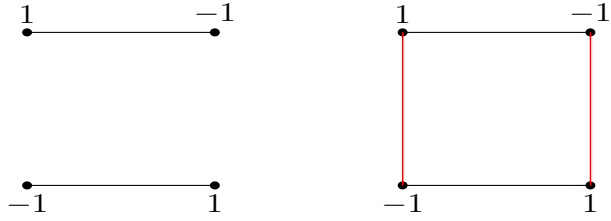}}}
\caption{\label{2reg} 
Construction of 2-regular bivalent graph $d=2, ~\lambda=4$ (right)
from the 1-regular bivalent graph $d=1,~ \lambda=2$ (left) 
by adding an alternate perfect matching.}
\end{figure}

The extension of three copies of chain of length 1 seen in Fig.\ref{bivalent} by 
adding an alternate perfect matching Theorem \ref{alt} (two times)
gives the $3$-regular bivalent graph shown on the right of Fig.\ref{3reg}

\begin{figure}[H]
\centerline{\resizebox{9 cm}{4 cm}{\includegraphics{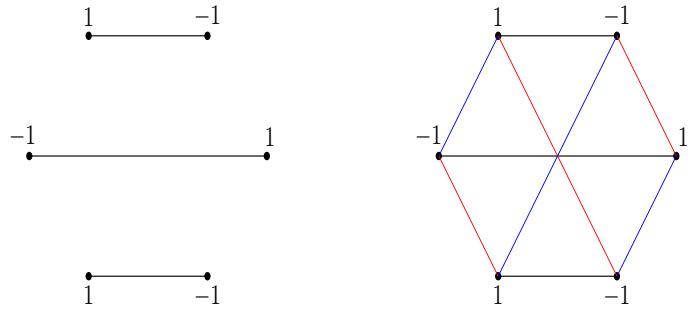}}}
\caption{\label{3reg}
Construction of 3-regular bivalent graph $d=3, ~\lambda=6$ (right)
from the 1-regular bivalent graph $d=1,~ \lambda=2$ (left) 
by adding two alternate perfect matchings.}
\end{figure}

Adding edges between equal nodes (Theorem \ref{L2S})
to three copies of a chain of length 1 seen in Fig.\ref{bivalent}
produces the bivalent eigenvector 
of the non-regular graphs shown in Fig.\ref{nonreg}
affording the same eigenvalue $\lambda=2$.

\begin{figure}[H]
\centerline{\resizebox{10 cm}{4 cm}{\includegraphics{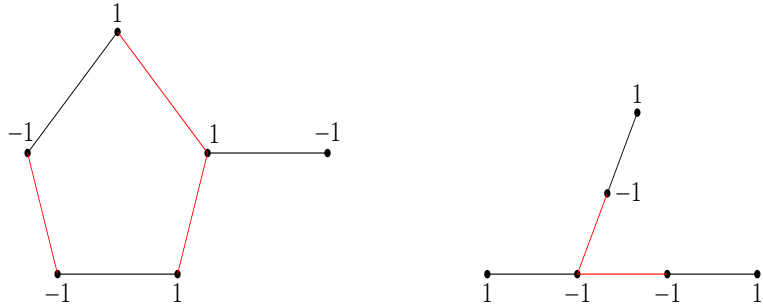}}}
\caption{\label{nonreg} Two bivalent graphs obtained from the 1-regular graph 
by adding edges between equal nodes, that afford (the same eigenvalue) $\lambda=2$.}
\end{figure}

More generally, note that a bivalent eigenvector affords an eigenvalue 
$\lambda\in \{0,2,4,\dots,2 d_{min} \}$ 
where $d_{min}$ is the smallest degree of nodes in the graph.

The following theorem was shown by Molitierno and Neumann \cite{molitierno03}.
We give here a different proof.

\begin{theorem}[Bivalent tree]
\label{tree}
A tree $\mathcal{T}$ is bivalent if and only if it has a perfect matching.
\end{theorem}
\begin{proof}
First note that a tree is bipartite 
and that a $1$-regular graph is a perfect matching.

Assume $\mathcal{T}$ be a bivalent tree. 
Then there exists an eigenvector $v$ with entries solely in $\{1,-1\}$
built from a $d$-regular bipartite graph by adding edges 
between nodes of equal values.
Since a tree always has leaves (nodes of degree 1), $d$ must be equal to $1$, 
the subgraph is $1$-regular hence a perfect matching.

Conversely, if a tree has a perfect matching,
it is easy to construct a bivalent eigenvector by taking opposite values in each 
edge of the matching, as there are no cycles in a tree, this can be done 
by Breadth-First Search (BFS) or Depth-First Search (DFS) algorithms.

\end{proof}

For a general graph, the existence of a perfect matching is not a sufficient condition to be bivalent. 
As examples, we show in Fig.\ref{asym} two asymmetric graphs i.e. which have no symmetries.

\begin{figure}[H]
\centerline{\resizebox{8 cm}{3 cm}{\includegraphics{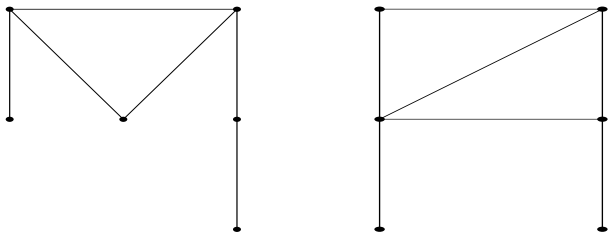}}}
\caption{\label{asym} Two asymmetric graphs of 6 nodes. They have a perfect matching but are not bivalent.}
\end{figure}

\section{Trivalent graphs}

The following theorem gives a characterization of trivalent graphs.
As noticed in the proof, soft regular graphs are trivalent.
In this section, we give examples of trivalent graphs obtained by 
the transformations of the theorem and also the transformation
of Theorem \ref{alt} (add/delete an alternate perfect matching).

\begin{theorem}[\textbf{Trivalent graphs}]
\label{trivalent}
Trivalent graphs are obtained from soft regular tripartite graphs
by applying to the same trivalent eigenvector the transformations :
\begin{itemize}
\item add a link between two equal nodes, 
\item extension by soft nodes
\item replace an edge by a soft square.
\end{itemize}
\end{theorem}

\begin{proof}
Let $\mathcal{G}$ be a graph having a trivalent eigenvector $v$ affording $\lambda$.

We reduce $\mathcal{G}$ by deleting all 
the edges between equal nodes (Theorem \ref{L2S})
and deleting soft nodes that are not adjacent to non-soft nodes 
(Theorem \ref{extension}),
thus obtaining a graph where edges only connect nodes with different values in $\{1,-1,0\}$.
This is a tripartite graph.

 For soft nodes $j$ in the reduced graph, the eigenvector condition
\bee
(d_j)(0)+\sum_{i\sim j,~v_i=1} (-1)(1)+\sum_{i\sim j,~v_i=-1} (-1) (-1) =(\lambda ) (0) =0,
\eee
requires that 
\bee
\mathrm{card}\left\{i\sim j,~v_i=+1\right\}=\mathrm{card}\left\{i\sim j,~v_i=-1\right\}.
\eee

 The eigenvector condition for nodes $j$ such that $v_j=1$, 
\bee
(d_j) (1)+\sum_{i\sim j,~v_i \neq 0} (-1) (-1)+\sum_{i\sim j,~v_i = 0} (-1) (0)=(\lambda ) (1).
\eee
A similar condition holds for nodes $j$ such that $v_j=-1$. 
Thus, 
\be
\label{lambda}
\lambda=d_j+\tilde{d}_j=2d_j-s_j,~~~  \forall j \in \mathcal{S}^c,
\ee
where $\mathcal{S}=\{k,~v_k=0\}$, 
$\mathcal{S}^c=\{1,\dots,N\}\backslash \mathcal{S}$, 
$\tilde{d}_j=\mathrm{card} \left\{i\sim j,~v_i\neq 0 \right\}$
and $s_j=\mathrm{card} \left\{i\sim j,~v_i=0 \right\}$.

The eigenvalue formula (\ref{lambda}) 
is satisfied for $\mathcal{G}$ being soft regular for $v$.
For trivalent graphs $\mathcal{G}$ that are not soft regular
(an example is shown on the left of Fig.\ref{nonSR}),
one can transform $\mathcal{G}$ to a soft regular graph
by applying Theorem \ref{square} several times 
and replacing each edge between nodes of opposite values 
by a square of two soft nodes
(as shown on the right of Fig.\ref{nonSR}), so that all
nodes $j$ in $\mathcal{S}^c$ verify $\tilde{d}_j=0$ and
$d_j=\lambda$.

Conversely, a soft regular tripartite graph $\mathcal{G}$ satisfies 
the eigenvalue condition (\ref{lambda}) and any extension
of $\mathcal{G}$ of the type above is trivalent.

\end{proof}

\begin{figure}[H]
\centerline{\resizebox{12 cm}{6 cm}{\includegraphics{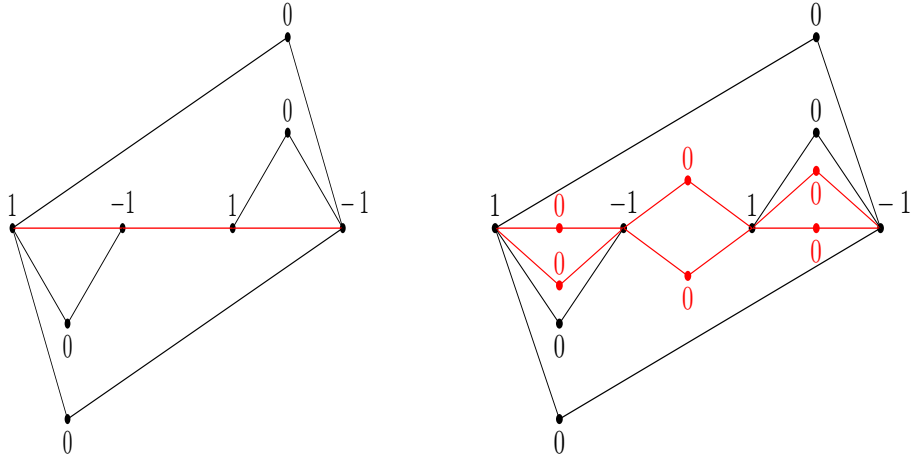}}}
\caption{\label{nonSR}
Two trivalent graphs, 
a non-soft regular graph (left) and a $5$-soft regular graph (right),
affording (the same eigenvalue) $\lambda=5$.}
\end{figure}

Below we give a classification by eigenvalues of the smallest trivalent graphs.
Then, the transformations connecting the elements within each class
generate trivalent graphs.

The smallest trivalent graph having eigenvalue 
$\lambda=d_j+\tilde{d}_j=1$ (where $j$ is a non-soft vertex) 
satisfies $d_j=1,~\tilde{d}_j=0$. 
That is the path on 3 nodes shown in Fig.\ref{lambda1}.

Trivalent trees are constructed from trivalent path on 3 vertices $(1,0,-1)^T$
by adding nodes between two equal-valued vertices Theorem \ref{L2S}
and extension of soft nodes Theorem \ref{extension}.
A characterization of all trees that have $1$ 
as the third smallest Laplacian eigenvalue is presented in \cite{barik}.

\begin{figure}[H]
\centerline{\resizebox{3 cm}{0.8 cm}{\includegraphics{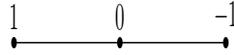}}}
\caption{\label{lambda1} The smallest trivalent graph affording $\lambda=1$.}
\end{figure}

The smallest trivalent graphs having eigenvalue 
$\lambda=d_j+\tilde{d}_j=2$ (where $j$ is a non-soft vertex)
satisfy :
\begin{itemize}
\item $d_j=2,~\tilde{d}_j=0$. That is the 4-cycle shown on the left of Fig.\ref{lambda2},
\item $d_j=\tilde{d}_j=1$. That is the 1-regular bivalent graph.
\end{itemize}

\begin{figure}[H]
\centerline{\resizebox{8 cm}{3 cm}{\includegraphics{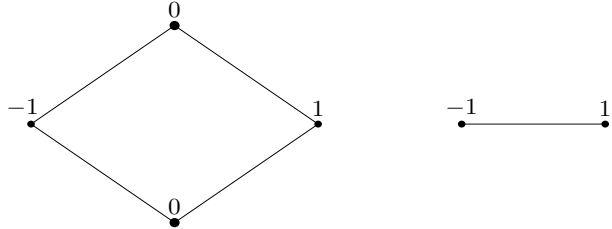}}}
\caption{\label{lambda2} The smallest trivalent graphs affording $\lambda=2$.}
\end{figure}

The smallest trivalent graphs having eigenvalue 
$\lambda=d_j+\tilde{d}_j=3$ (where $j$ is a non-soft vertex)
satisfy :
\begin{itemize}
\item $d_j=3,~\tilde{d}_j=0$. That is the graph shown on the left of Fig.\ref{lambda3},
\item $d_j=2,~\tilde{d}_j=1$. That is the graph shown on the right of Fig.\ref{lambda3}.
\end{itemize}

\begin{figure}[H]
\centerline{\resizebox{8 cm}{4 cm}{\includegraphics{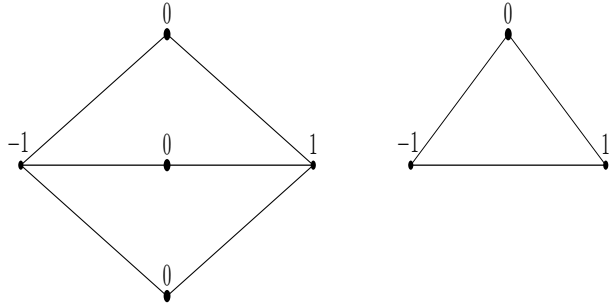}}}
\caption{\label{lambda3} The smallest trivalent graphs affording $\lambda=3$.}
\end{figure}

The smallest trivalent graphs having eigenvalue 
$\lambda=d_j+\tilde{d}_j=4$ (where $j$ is a non-soft vertex)
satisfy :
\begin{itemize}
\item $d_j=4,~\tilde{d}_j=0$. That is the graph shown on the left of Fig.\ref{lambda4},
\item $d_j=3,~\tilde{d}_j=1$. That is the graph shown on the middle of Fig.\ref{lambda4},
\item $d_j=\tilde{d}_j=2$. That is the 2-regular bivalent graph (right of Fig.\ref{lambda4}).
\end{itemize}

\begin{figure}[H]
\centerline{\resizebox{12 cm}{4 cm}{\includegraphics{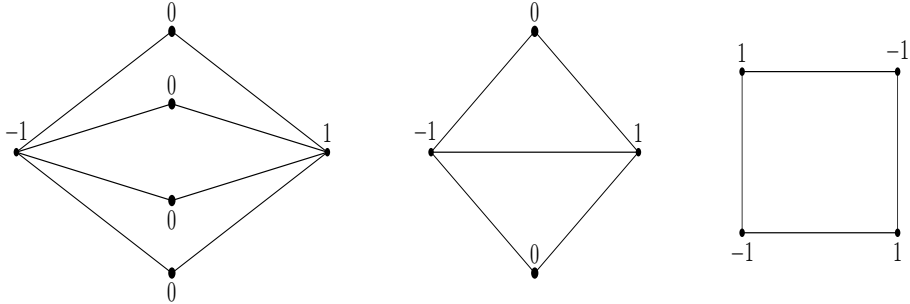}}}
\caption{\label{lambda4} The smallest trivalent graphs affording $\lambda=4$.}
\end{figure}

The smallest trivalent graphs having eigenvalue 
$\lambda=d_j+\tilde{d}_j=5$ (where $j$ is a non-soft vertex)
satisfy :
\begin{itemize}
\item $d_j=5,~\tilde{d}_j=0$. That is the graph shown on the left of Fig.\ref{lambda5},
\item $d_j=4,~\tilde{d}_j=1$. That is the graph shown on the middle of Fig.\ref{lambda5},
\item $d_j=3,~\tilde{d}_j=2$. That is the graph shown on the right of Fig.\ref{lambda5}.
\end{itemize}

\begin{figure}[H]
\centerline{\resizebox{13 cm}{5 cm}{\includegraphics{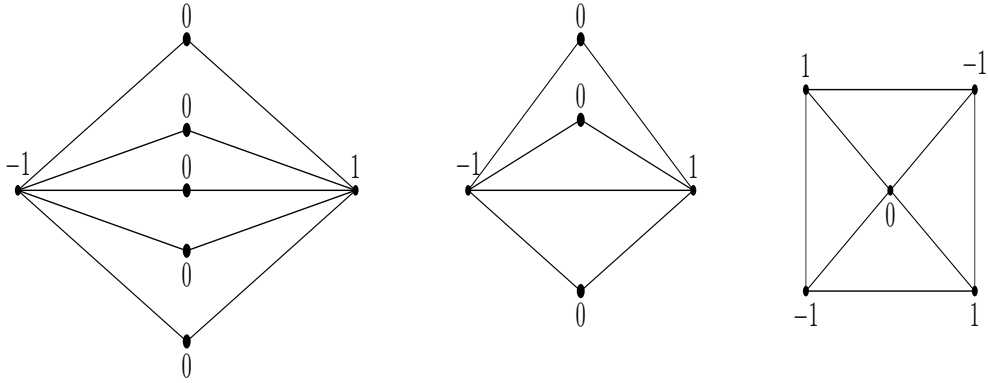}}}
\caption{\label{lambda5} The smallest trivalent graphs affording $\lambda=5$.}
\end{figure}

The smallest trivalent graphs having eigenvalue 
$\lambda=d_j+\tilde{d}_j=6$ (where $j$ is a non-soft vertex)
satisfy :
\begin{itemize}
\item $d_j=6,~\tilde{d}_j=0$. That is the first graph in Fig.\ref{lambda6},
\item $d_j=5,~\tilde{d}_j=1$. That is the second graph in Fig.\ref{lambda6},
\item $d_j=4,~\tilde{d}_j=2$. That is the third graph in Fig.\ref{lambda6},
\item $d_j=\tilde{d}_j=3$. That is the 3-regular bivalent graph (right of Fig.\ref{lambda6}).
\end{itemize}

\begin{figure}[H]
\centerline{\resizebox{14 cm}{6 cm}{\includegraphics{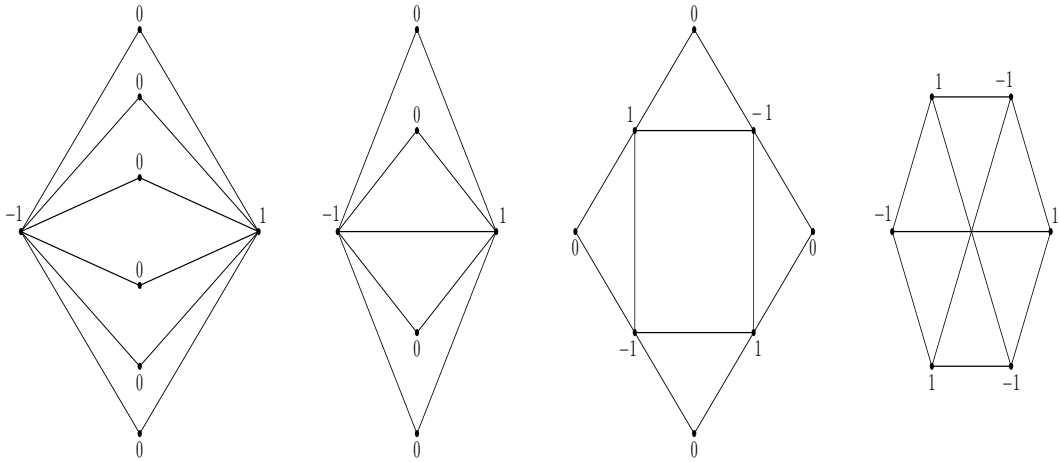}}}
\caption{\label{lambda6} The smallest trivalent graphs affording $\lambda=6$.}
\end{figure}

\section{Conclusion}
We have characterized bivalent and trivalent graphs by applying
Laplacian eigenvector transformations;  these
are links between two equal nodes, replacing an edge by a soft square,
and adding or deleting an alternate perfect matching.
We show that bivalent graphs are the regular bipartite graphs and 
their extensions obtained by
adding edges between two equal nodes.
We define a soft regular graph 
as having a Laplacian eigenvector with soft nodes
such that each non-soft node has the same degree.
Trivalent graphs are shown to be the soft regular graphs and their extensions.
However, the question of whether a given graph is bivalent or trivalent, 
is difficult and remains open.
The exploration of these graphs is important 
for nonlinear dynamical systems on networks.

\section*{Acknowledgment}
This work is part of the XTerM project, co-financed by the European Union 
with the European regional development fund (ERDF) and by the Normandie Regional Council.

\end{document}